\documentclass[a4paper]{amsart}

\usepackage{amsfonts}
\usepackage{amssymb}
\usepackage{amsmath}
\usepackage{hyperref}
\usepackage{mathrsfs}
\usepackage{centernot}
\usepackage{mathdots}
\usepackage{stmaryrd}
\usepackage[all]{xy}

\usepackage{color}

\usepackage[foot]{amsaddr}   

\newcommand{\suchthat}{\,:\,}
\newcommand{\where}{\,|\,}

\newcommand{\quo}[1]{\overline{#1}}
\newcommand{\units}[1]{{#1^\times}}

\newcommand{\veps}{\varepsilon}
\newcommand{\vphi}{\varphi}

\newcommand{\nMat}[2]{\mathrm{M}_{#2}(#1)}

\newcommand{\SO}{{\mathrm{SO}}}
\newcommand{\uSO}{{\mathbf{SO}}}
\newcommand{\uO}{{\mathbf{O}}}

\newtheorem{thm}{Theorem} 
\newtheorem{lem}[thm]{Lemma}

\newtheorem{que}[thm]{Question}

\newtheoremstyle{roman} 
    {8.0pt plus 2.0pt minus 4.0pt}                    
    {8.0pt plus 2.0pt minus 4.0pt}                    
    {\normalfont}                
    {}                           
    {\bfseries}                  
    {.}                          
    {5pt plus 1pt minus 1pt}     
    {}  

\theoremstyle{roman}

\newtheorem{example}[thm]{Example}

\theoremstyle{plain}

\newcommand{\frakm}{{\mathfrak{m}}}

\DeclareMathOperator{\Br}{Br} %
\DeclareMathOperator{\End}{End} %
\DeclareMathOperator{\Hom}{Hom} %
\DeclareMathOperator{\id}{id} %
\DeclareMathOperator{\Jac}{Jac} %

\DeclareMathOperator{\Max}{Max}
\DeclareMathOperator{\Nrd}{Nrd} %
\newcommand{\op}{\mathrm{op}} %
\DeclareMathOperator{\rank}{rank}
\DeclareMathOperator{\Skew}{Skew} %
\DeclareMathOperator{\Spec}{Spec} %
\title[On The Orthogonal Group]{On The Non-Neutral Component of Outer Forms of The Orthogonal Group}

\author{Uriya A.\ First$^*$}
\date{\today}
\address{$^*$University of Haifa}
\email{uriya.first@gmail.com}

\begin{document}

\maketitle

Let $A$ be a central simple algebra over a field 
$F$ of characteristic not $2$,
and let $\sigma:A\to A$ be an orthogonal involution (see \cite[Ch.~I]{Knus_1998_book_of_involutions}
for the definitions). Let $\mathrm{O}(A,\sigma)$ denote the group
of elements $u\in A$ with $u^\sigma u=1$. The reduced norm
map  $\Nrd_{A/F}:A\to F$
restricts to a group homomorphism, 
$
\Nrd_{A/F}:\mathrm{O}(A,\sigma)\to \{\pm 1\}
$; its kernel,
$\SO(A,\sigma)$, is the special orthogonal group of $(A,\sigma)$. 
Both  $\mathrm{O}(A,\sigma)$ and $\SO(A,\sigma)$ can be regarded
as the $F$-points of algebraic groups over $F$, denoted 
$\uO(A,\sigma)$ and $\uSO(A,\sigma)$, respectively.

The question of whether $
\mathrm{O}(A,\sigma) 
$ has elements of reduced norm $-1$
is equivalent to asking whether the non-neutral
component of the algebraic group   $\uO(A,\sigma)$, which is 
an  $\uSO(A,\sigma)$-torsor, has an $F$-point. 
It is well-known
that such an $F$-point exists  if and only if $[A]$, the Brauer class of $A$,
is trivial in the Brauer group of $F$, denoted $\Br F$;
see \cite[Lemma~2.6.1b]{Kneser_1969_Galois_cohomology_classical_groups}, for instance.

In this note, we generalize this result to   Azumaya
algebras with orthogonal involutions over semilocal commutative rings.
Given a commutative ring $R$, recall
that an $R$-algebra $A$ is called Azumaya if $A$
is a finitely generated projective $R$-module
and $A(\frakm):=A\otimes_R (R/\frakm)$
is a central simple $(R/\frakm)$-algebra for every maximal ideal
$\frakm\in \Max R$. In this case, an $R$-linear involution $\sigma:A\to A$
is called orthogonal if its specialization $\sigma(\frakm):=\sigma\otimes_R \id_{R/\frakm}$
is orthogonal for all $\frakm\in \Max R$. 
See \cite[III.\S5, III.\S8]{Knus_1991_quadratic_hermitian_forms} or \cite{Ford_2017_separable_algebras} for 
an extensive discussion. Note also that $\Nrd_{A/R}$
takes $\mathrm{O}(A,\sigma)$ to $\mu_2(R):=\{\veps\in R\suchthat \veps^2=1\}$.
We prove:

\begin{thm}\label{TH:main-I}
	Let $R$ be a commutative semilocal ring with $2\in\units{R}$,
	let $A$ be an Azumaya $R$-algebra and let $\sigma:A\to A$ be an orthogonal involution.
	Then $\mathrm{O}(A,\sigma)$ contains elements of reduced norm $-1$
	if and only if $[A]=0$
	in $\Br R$.
\end{thm}

In the process, we prove another result  
of independent interest:

\begin{thm}\label{TH:main-II}
	Let $R$, $A$ and $\sigma$ be as in Theorem~\ref{TH:main-I}. 
	Then the natural map 
	\[
	\SO(A,\sigma)\to \prod_{\frakm\in \Max R}\SO(A(\frakm),\sigma(\frakm))
	\]
	is surjective.
\end{thm}

This   was proved
by Knebusch \cite[Satz~0.4]{Knebusch_1969_isometries_over_semilocal_rings} when 
$A$ is a matrix algebra over $R$. The surjectivity of 
$\mathrm{O}(A,\sigma)\to \prod_{\frakm\in \Max R}\mathrm{O}(A(\frakm),\sigma(\frakm))$
may fail under the same assumptions (Example~\ref{EX:not-surjective}).

\medskip

Applications of both theorems to Witt
groups of Azumaya algebras with involution  appear in   \cite{First_2018_octagon_preprint}.

\medskip

We show that the ``if'' part of Theorem~\ref{TH:main-I} is false
if $R$ is not assumed to be semilocal, see Example~\ref{EX:counter}.
As for the ``only if'' part of Theorem~\ref{TH:main-I}, we ask:

\begin{que}\label{QE:one}
	Let $R$ be a commutative ring with $2\in\units{R}$, let $A$ be an Azumaya
	$R$-algebra, and let $\sigma:A\to A$ be an orthogonal involution.
	Suppose that $\mathrm{O}(A,\sigma)$  contains elements of reduced
	norm $-1$.
	Is it the case that $[A]=0$?
\end{que}

We expect that the answer is ``yes''.
By Theorem~\ref{TH:main-I},
a counterexample, if it exists, will have the remarkable property
that $[A]\neq 0$ while $[A\otimes_RS]=0$ in $\Br S$ for every
semilocal commutative $R$-algebra $S$. We do not know if Azumaya
algebras with this property exist. (Ojanguren \cite{Ojanguren_1974_locally_triv_Az_alg}
gave an example having this property
for any local $S$, but in his example, $[A\otimes_RS]$ remains nontrivial
if $S$ is taken to be the localization of $R$ away from three  particular  prime ideals.)
We further note that the answer to Question~\ref{QE:one} is ``yes''
when $R$ is a regular domain.
Indeed, writing $F$ for the fraction field of $R$,
we observed that $[A\otimes_RF]=0$ in $\Br F$,
and the map $\Br R\to \Br F$
is injective by  the Auslander--Goldman theorem \cite[Theorem~7.2]{Auslander_1960_Brauer_Group}.

\section{Proof of Theorem~\ref{TH:main-II}}

We shall derive Theorem~\ref{TH:main-II} from a more general
theorem addressing semilocal rings with involution.
Recall 
that a   ring $A$
is called semilocal if $A/\Jac A$ is semisimple
aritinian, where $\Jac A$ denotes the Jacobson radical of $A$.

\medskip

Let $(A,\sigma)$ be a ring with
involution such that $2\in\units{A}$.
We let $\Skew(A,\sigma)=\{a\in A\suchthat a^\sigma=-a\}$. 
Given $y\in A$ and $a\in\Skew(A,\sigma)$ such that
$\frac{1}{2}y^\sigma y+a\in\units{A}$,
define
\[
s_{y,a}=1-y (\textstyle{\frac{1}{2}}y^\sigma y+a)^{-1}y^\sigma \in A.
\]

Consider the $(\sigma,1)$-hermitian form $f_1:A\times A\to A$
given by $f_1(x,y)=x^\sigma y$; here, $A$ is viewed as a right module
over itself.
Identifying $\End_A(A)$ with $A$ via $\vphi\mapsto \vphi(1_A)$,
one easily checks that the isometry group of $f_1$ is 
\[
U(A,\sigma):=\{u\in A\suchthat u^\sigma u=1\}.
\]
Moreover, the elements $  s_{y,a} $ are precisely the $1$-reflections
of $f_1$ in the sense of \cite[\S1]{Reiter_1975_Witts_extension_theorem} or 
\cite[\S3]{First_2015_Witts_extension_theorem}.
Thus,  
$s_{y,a}\in U(A,\sigma)$ for all $y$ and $a$ as above
(\cite[Proposition~1.3]{Reiter_1975_Witts_extension_theorem}
or
\cite[Proposition~3.3]{First_2015_Witts_extension_theorem}).
This can also be checked   by computation.

\medskip

Suppose that $A$ is semisimple artinian.
We define a subgroup $U^0(A,\sigma)$ of $U(A,\sigma)$ as follows:
Assume first that $A$ is simple artinian. By the Artin--Wedderburn theorem, 
$A\cong \nMat{D}{n}$,
where $D$ is a division ring with center $K$. We then define
\[
U^0(A,\sigma):=\left\{\begin{array}{ll}
\mathrm{SO}(A,\sigma) &\text{$D=K$ and $\sigma$ is orthogonal}\\
U(A,\sigma)& \text{otherwise}
.
\end{array}
\right.
\]
Next, if $A$ is not simple but $(A,\sigma)$ is simple
as a ring with involution, then there is a simple
artinian ring $B$ and an isomorhism $A\cong B\times B^\op$ 
under which $\sigma$ corresponds to $(x,y^\op)\mapsto (y,x^\op)$ ($x,y\in B$). We
then set
\[
U^0(A,\sigma):=U(A,\sigma).
\]
Finally, when $A$ is an arbitrary semisimple artinian ring,
there exists an essentially unique factorization
$(A,\sigma)=\prod_{i=1}^t (A_i,\sigma_i)$
such that each factor $(A_i,\sigma_i)$ fits into exactly 
one of the previous two cases
(see 
\cite[p.~486]{Reiter_1975_Witts_extension_theorem},
for instance).
We then  define
\[
U^0(A,\sigma)=\prod_{i=1}^t U^0(A_i,\sigma_i).
\]

\begin{example}\label{EX:orth-inv-U-zero}
	Suppose that $A$ is an Azumaya algebra
	over a finite product of fields $R=\prod_{i=1}^t K_i$
	and $\sigma:A\to A$ is an orthogonal involution.
	Then $U^0(A,\sigma)=\SO(A,\sigma)$.
	Indeed, writing $(A,\sigma)=\prod_i (A_i,\sigma_i)$ with $A_i$ a central
	simple $K_i$-algebra,
	we reduce into checking that $U^0(A_i,\sigma_i)=\SO(A_i,\sigma_i)$.
	This follows from the definition if $[A_i]=0$ (in $\Br K_i$),
	and from \cite[Lemma~2.6.1b]{Kneser_1969_Galois_cohomology_classical_groups} 	
	if $[A_i]\neq 0$.
\end{example}

\begin{thm}\label{TH:reflection-generation}
	Let $(A,\sigma)$ be a semisimple artinian ring
	with involution such that $2\in\units{A}$. Then
	the subgroup of $U(A,\sigma)$ generated
	by the elements $s_{y,a}$ with 
	$y\in A$, $a\in\Skew(A,\sigma)$ and
	$\frac{1}{2}y^\sigma y+a\in\units{A}$
	contains $U^0(A,\sigma)$.
\end{thm}

\begin{proof}
	We observed above that the elements
	$s_{y,a}$ are precisely the reflections
	of a $(\sigma,1)$-hermitian form $f_1:A\times A\to A$.
	The  theorem is therefore a special case of 
	\cite[Theorem~5.8(ii)]{First_2015_Witts_extension_theorem}
	(see also Remark~2.1 in that source).
\end{proof}

\begin{thm}\label{TH:U-zero-onto}
	Let $(A,\sigma)$ be a semilocal ring with involution
	such that $2\in\units{A}$.
	Write $\quo{A}=A/\Jac A$, denote
	the quotient map $A\to \quo{A}$ by $a\mapsto \quo{a}$
	and let $\quo{\sigma}:\quo{A}\to \quo{A}$
	be given by $\quo{a}^{\quo{\sigma}}=\quo{a^\sigma}$.
	Then the image of the   map
	\[ u\mapsto \quo{u}:U(A,\sigma)\to U (\quo{A},\quo{\sigma})\]
	contains $U^0(\quo{A},\quo{\sigma})$. 
\end{thm}

\begin{proof} 
	By Theorem~\ref{TH:reflection-generation},
	every element of $U^0(\quo{A},\quo{\sigma})$
	is a product of elements of the form $s_{\tilde y,\tilde a}$
	with $\tilde y\in \quo{A}$, $\tilde a\in\Skew(\quo{A},\quo{\sigma})$,
	$\frac{1}{2}\tilde{y}^\sigma \tilde{y}+\tilde{a}\in\units{\quo{A}}$.
	It is therefore   enough to prove that there exists $u\in U(A,\sigma)$
	with $\quo{u}=s_{\tilde y,\tilde a}$.
	Choose   $y,a\in A$ with $\quo{y}=\tilde{y}$ and $\quo{a}=\tilde{a}$.
	Replacing $a$ with $\frac{1}{2}(a-a^\sigma)$,
	we may assume that $a\in \Skew(A,\sigma)$.
	Since $\quo{\frac{1}{2}y^\sigma y+a}=\frac{1}{2}\tilde{y}^\sigma \tilde{y}+\tilde{a}\in \units{(\quo{A})}$,
	we have $\frac{1}{2}y^\sigma y+a\in\units{A}$.
	We may therefore take $u:=s_{y,a}$, which clearly
	satisfies $\quo{u}=s_{\tilde{y},\tilde{a}}$.
\end{proof}

Now we can prove Theorem~\ref{TH:main-II}.

\begin{proof}[Proof of Theorem~\ref{TH:main-II}]
	Write $J=\Jac R$.
	We first observe that
	$\Jac A=JA $. Indeed,
	$A/J A\cong A\otimes (R/J)$
	is Azumaya over $R/J$, which is a product of fields,
	so $A /JA$ is semisimple artinian, meaning that
	$JA\supseteq \Jac A$. On the other
	hand $JA \subseteq \Jac A$ because $A$ is finitely
	generated as an $R$-module \cite[Ch.~II, Corollary~4.2.4]{Knus_1991_quadratic_hermitian_forms}.
	
	Now, using the notation of Theorem~\ref{TH:U-zero-onto},
	$\quo{A}=A/\Jac A\cong A\otimes (R/\Jac R)=A\otimes \prod_{i=1}^t (R/\frakm_i)\cong
	\prod_{i=1}^t A(\frakm_i)$, so we may identify
	$\prod_{i=1}^t A(\frakm_i)$ with $\quo{A}$. 
	Under this identification, $\prod_{i}\sigma(\frakm_i)$ corresponds
	to $\quo{\sigma}$, so   
	we need to prove that the natural map $u\mapsto \quo{u}:\SO(A,\sigma)\to
	\SO(\quo{A},\quo{\sigma})$ is surjective.
	
	Let  $v\in \SO(\quo{A},\quo{\sigma})$.
	By Theorem~\ref{TH:U-zero-onto} and Example~\ref{EX:orth-inv-U-zero},
	there
	exists $u\in \mathrm{O}(A,\sigma)=U(A,\sigma)$
	such that $\quo{u}=v$.
	We claim that $u\in \SO(A,\sigma)$.
	Indeed, write $\alpha=\Nrd_{A/R}(u)\in R$.
	Then $\alpha^2=1$, and so  $\frac{1}{2}(1-\alpha)$ is an idempotent.
	Since $\Nrd(v)= {1}$ in $R/J$, the image of $\frac{1}{2}(1-\alpha)$
	in $R/J$ is $\frac{1}{2}(1-1)=0$. As $J$ contains no nonzero
	idempotents, it follows that $\frac{1}{2}(1-\alpha)=0$, or rather, $\alpha=1$.
\end{proof}

\begin{example}\label{EX:not-surjective}
	The assumptions of Theorem~\ref{TH:main-II}
	do not guarantee that $\mathrm{O}(A,\sigma)\to \prod_{\frakm\in \Max R}\mathrm{O}(A(\frakm),\sigma(\frakm))$
	is surjective in general.
	As a trivial counterexample one could take $R$ to be any non-local semilocal domain and note
	that $\mathrm{O}(R,\id)=\{\pm1\}$ while 
	$|\prod_{\frakm\in \Max R}\mathrm{O}(R(\frakm),\id_{R(\frakm)})|>2$.
	A counterexample with local $R$ can be constructed as follows.
	Take take $R$ to be the localization of $\mathbb{Z}$ at $5\mathbb{Z}$,  
	let $A$ be the quaternion Azumaya algebra $R\langle i,j\where i^2=j^2=-1,\,ij=-ji\rangle$
	and let $\sigma:A\to A$ be the orthogonal involution   fixing $i$ and $j$.
	Let $\frakm=5R$ denote the   maximal ideal of $R$ and let
	$v$ be the image of $3i$ in $A(\frakm)$.
	One readily checks that  $v\in \mathrm{O}(A(\frakm),\sigma(\frakm))$ and $\Nrd(v)=-1$.
	However,  $v$  cannot be lifted to an element of $\mathrm{O}(A,\sigma)$.
	Indeed, if such a lift existed, it would have reduced norm $-1$, 
	but 
	one can check directly that   elements of $A$ have non-negative reduced norms.	
%
\end{example}


\section{Proof of Theorem~\ref{TH:main-I}}

	\begin{lem}\label{LM:existence-of-good-char-pol}
		Let $A$ be an Azumaya algebra of constant degree $d$ over a semilocal
		ring $R$ with $2\in\units{R}$ and let $\sigma:A\to A$
		be an orthogonal involution.
		If $[A]=0$, then there exists
		$u\in \mathrm{O}(A,\sigma)$ 
		with   $u^2=1$ and reduced characteristic polynomial
		$(t+1)(t-1)^{d-1}$. In particular,
		$\Nrd_{A/R}(u)=-1$.
	\end{lem}

	\begin{proof}
		Since $[A]=0$,
		we may assume that $A=\End_R(Q)$ for some finitely generated projective $R$-module $Q$
		of rank $d$.	
		By \cite[Theorem~4.2a]{Saltman_1978_Azumaya_algebras_w_involution}
		(or, alternatively, \cite[Proposition~4.6]{First_2015_morita_equiv_to_op}),
		there exist    $\delta\in \mu_2(R) $,
		a rank-$1$ projective $R$-module  $L$,
		and a unimodular   $L$-valued bilinear form
		$g:Q\times Q\to L$ satisfying $g(x,y)=\delta  g(y,x) $
		and $g(ax,y)=g(x,a^\sigma y)$ 
		for all $x,y\in Q$, $a\in A$.
		(Here, unimodularity means that $x\mapsto g(x,-):Q\to \Hom_R(Q,L)$
		is bijective.) Since $R$ is semilocal and $L$ has rank $1$, $L\cong R$,
		so we may assume $L=R$. 
		Moreover, $\delta=1$ because $\sigma$ is orthogonal; see
		\cite[p.~170]{Knus_1991_quadratic_hermitian_forms}, for instance.	
        Now, choose a vector $x\in Q$ with $g(x,x)\in\units{R}$ ---
        to see its existence, check it modulo $\Jac R$
        and take an arbitrary lift.
        Writing $P=\{y\in Q\suchthat g(x,y)=0\}$, we have
        $Q=xR\oplus P$ and $\rank P=d-1$. 
        The reflection  $u:=(-\id_{xR})\oplus \id_P$
        is the required   element.
	\end{proof}

	\begin{proof}[Proof of Theorem~\ref{TH:main-I}]
		By writing $R$ as a product of connected semilocal commutative rings
		and working over each factor separately,
		we may assume that $R$ is connected. Thus, $d:=\deg A$ is constant
		on $\Spec R$.
	
%
%
		That $[A]=0$ implies the surjectivity of $\Nrd_{A/R}:\mathrm{O}(A,\sigma)\to\{\pm 1\}$
		has been verified in Lemma~\ref{LM:existence-of-good-char-pol},
		so we turn to prove the converse.
		
		Let $u\in \mathrm{O}(A,\sigma)$ be an element
		with $\Nrd_{A/R}(u)=-1$.
		We let $\frakm_1,\dots,\frakm_t$
		denote the maximal ideals of $R$
		and set  $A_i=A(\frakm_i)$, $\sigma_i=\sigma(\frakm_i)$.
		We further let $u_i$ denote the image of $u$ in $A_i$.
		
		Since $A$ carries
		an involution fixing $R$, we have $[A]=[A^\op]=-[A]$.
		By a theorem of Saltman  \cite{Saltman_1981_Brauer_group_is_torsion},
		we also have $d  [A]=0$ in $\Br R$,
		so $[A]=0$ if $d$ is odd. We may therefore assume that $d$ is even.

		If there exists $1\leq i\leq t$ such
		that $[A_i]\neq 0$ in $\Br (R/\frakm_i)$,
		then $\Nrd(u_i)=1$ by \cite[Lemma~2.6.1b]{Kneser_1969_Galois_cohomology_classical_groups},
		which is impossible (because $2\in\units{R}$).
		Thus, $[A_i]=0$ for all $i$.
		Now, by Lemma~\ref{LM:existence-of-good-char-pol},
		for every $1\leq i\leq t$,
		there is $v_i\in \mathrm{O}(A_i,\sigma_i)$
		with reduced characteristic polynomial equal to $(t+1)(t-1)^{d-1}$.
		
		Observe that $\Nrd(u_i^{-1}v_i)=1$, hence
		$u_i^{-1}v_i\in \SO(A_i,\sigma_i)$.
		By Theorem~\ref{TH:main-II}, there exists
		$w\in \SO(A,\sigma)$
		such that the image of $w $ in $A_i$
		is $u_i^{-1}v_i$ for all $i$. Writing $v:=uw\in \mathrm{O}(A,\sigma)$,
		we see that the image of $v$ in $A_i$ is $v_i$ for all $i$.

		Let $f=f_v \in R[t]$
		denote the reduced characteristic polynomial of $v$.
		Then  
		\[f\equiv (t+1)({t-1})^{d-1} \bmod \frakm_i \]
		for all $1\leq i\leq t$.
		Note that $\sigma:A\to A$
		preserves the reduced characteristic polynomial
		(use \cite[III.\S8.5]{Knus_1991_quadratic_hermitian_forms}
		to see that
		there is a faithfully flat commutative $R$-algebra $S$
		such that  $(A\otimes S,\sigma\otimes\id_S)$ is  isomorphic to $\nMat{S}{d}$
		with the transpose involution).
		Since $v^{-1}=v^\sigma$,
		this means that $f_{v^{-1}}=f_{v^\sigma}=f$,
		so $f(0)^{-1} t^d f(t^{-1})= f$ in $R[t,t^{-1}]$.
		Substituting $t=-1$ and noting that $f(0)= \Nrd_{A/R}(v)=-1$ because
		$d$ is even, we get $-f(-1)=f(-1)$, hence $f(-1)=0$
		and   $t+1\mid f$.
		
		Write $f =(t+1)g $
		and $g =(t+1)r +\alpha$, where $g ,r \in R[t]$  and  
		$\alpha=g(-1)$.
		Since $f\equiv (t+1)(t-1)^{d-1} \bmod \frakm_i $,
		we have $g\equiv (t-1)^{d-1} \bmod\,\frakm_i $
		and  $\alpha \equiv (-2)^{d-1} \bmod \frakm_i $.
		As this holds for all $i$,  we have $\alpha  \in\units{R}$.
		Thus, 
		\[\alpha^{-1}g-\alpha^{-1}(t+1) r=\alpha^{-1}\alpha=1.\] 
		
		Put $e=\alpha^{-1}g(v)$ and $e'= -\alpha^{-1}(v +1)r(v)$.
		Then $e+e'=1_A$ and
		$ee'=e'e=0$ (because   $(v+1)g(v)=f(v)=0$). Thus, $e=e(e+e')=e^2$.
		Let $e_i$ denote the image of $e$ in $A_i$.
		Then  $e_i=(-2)^{1-d}g(v_i)=(-2)^{1-d}(v_i-1)^{d-1}$ 
		has rank one. This means that 
		$eAe$ is a projective $R$-algebra
		of rank $1$, so $eAe\cong R$. 
		Since $eAe\cong\End_A(eA)$, we have $[A]=[eAe]=[R]=0$
		\cite[Proposition~III.5.3.1]{Knus_1991_quadratic_hermitian_forms}.
\end{proof}


\begin{example}\label{EX:counter}
	The ``if'' part of Theorem~\ref{TH:main-I}
	is false if $R$ is not assumed to be semilocal.
	Indeed, take $R$ to be an integral domain with $2\in\units{R}$
	admitting a   non-principal
	invertible fractional ideal $L$ (we view
	$L$ as a subset of the fraction field of $R$).  Define
	\[
	A=\left[\begin{array}{cc} R & L^{-1} \\ L  & R\end{array}\right]
	\]
	and let $\sigma:A\to A$ be the involution given by
	$
	\left[\begin{smallmatrix} a & b \\ c & d \end{smallmatrix}\right]^\sigma=
	\left[\begin{smallmatrix} d & b \\ c & a \end{smallmatrix}\right]
	$.
	To see that $A$ is Azumaya over $R$ and $\sigma$ is orthogonal,
	observe that there is an isomorphism $A\cong \End_R(R\oplus L)$ under which
	$\sigma$ is the adjoint to   the unimodular symmetric bilinear form
	$f:(R\oplus L)\times (R\oplus L)\to L$ given by 
	$f(\left[\begin{smallmatrix} r_1 \\ \ell_1 \end{smallmatrix}\right],
	\left[\begin{smallmatrix} r_2 \\ \ell_2 \end{smallmatrix}\right])=r_1\ell_2+r_2\ell_1$.
	This also shows that $[A]=0$ in $\Br R$.
	
	Straightforward computation shows that elements of $\mathrm{O}(A,\sigma)$
	of determinant $-1$   are
	of the form $\left[\begin{smallmatrix} 0 & x^{-1} \\ x & 0 \end{smallmatrix}\right]$,
	where $x\in L$. If such an element exists, then
	$x^{-1}R\subseteq L^{-1}$, or rather, $L\subseteq xR$. Since $x\in L$,
	this means that $L=xR$, contradicting our assumption that $L$ is not principal.
	Thus, 	$\mathrm{O}(A,\sigma)=\SO(A,\sigma)$
	and $\Nrd_{A/R}:\mathrm{O}(A,\sigma)\to \mu_2(R)$ is not surjective.
	
	We remark that if $R\oplus L\cong M\oplus M$ for some invertible fractional
	ideal $M$, then we also have $A\cong \End_R(M\oplus M)\cong\nMat{R}{2}$.
	Such examples exist, e.g., take $R$ to be a Dedekind domain
	with class group containing an element $[M]$ of order $4$
	(use \cite{Eakin_1973_class_groups}, for instance)
	and let $L=M^2$. 
\end{example}

\bibliographystyle{plain}
\bibliography{MyBib_18_05}

\begin{thebibliography}{10}

\bibitem{Auslander_1960_Brauer_Group}
Maurice Auslander and Oscar Goldman.
\newblock The {B}rauer group of a commutative ring.
\newblock {\em Trans. Amer. Math. Soc.}, 97:367--409, 1960.

\bibitem{Eakin_1973_class_groups}
Paul Eakin and W.~Heinzer.
\newblock More noneuclidian {${\rm PID}$}'s and {D}edekind domains with
  prescribed class group.
\newblock {\em Proc. Amer. Math. Soc.}, 40:66--68, 1973.

\bibitem{First_2015_morita_equiv_to_op}
Uriya~A. First.
\newblock Rings that are {M}orita equivalent to their opposites.
\newblock {\em J. Algebra}, 430:26--61, 2015.

\bibitem{First_2015_Witts_extension_theorem}
Uriya~A. First.
\newblock Witt's extension theorem for quadratic spaces over semiperfect rings.
\newblock {\em J. Pure Appl. Algebra}, 219(12):5673--5696, 2015.

\bibitem{First_2018_octagon_preprint}
Uriya~{A}. First.
\newblock An $8$-periodic exact sequence of {W}itt groups of {A}zumaya algebras
  with involution.
\newblock 2019.
\newblock Preprint.

\bibitem{Ford_2017_separable_algebras}
Timothy~J. Ford.
\newblock {\em Separable algebras}, volume 183 of {\em Graduate Studies in
  Mathematics}.
\newblock American Mathematical Society, Providence, RI, 2017.

\bibitem{Knebusch_1969_isometries_over_semilocal_rings}
Manfred Knebusch.
\newblock Isometrien \"{u}ber semilokalen {R}ingen.
\newblock {\em Math. Z.}, 108:255--268, 1969.

\bibitem{Kneser_1969_Galois_cohomology_classical_groups}
M.~Kneser.
\newblock {\em Lectures on {G}alois cohomology of classical groups}.
\newblock Tata Institute of Fundamental Research, Bombay, 1969.
\newblock With an appendix by T. A. Springer, Notes by P. Jothilingam, Tata
  Institute of Fundamental Research Lectures on Mathematics, No. 47.

\bibitem{Knus_1991_quadratic_hermitian_forms}
Max-Albert Knus.
\newblock {\em Quadratic and {H}ermitian forms over rings}, volume 294 of {\em
  Grundlehren der Mathematischen Wissenschaften [Fundamental Principles of
  Mathematical Sciences]}.
\newblock Springer-Verlag, Berlin, 1991.
\newblock With a foreword by I. Bertuccioni.

\bibitem{Knus_1998_book_of_involutions}
Max-Albert Knus, Alexander Merkurjev, Markus Rost, and Jean-Pierre Tignol.
\newblock {\em The book of involutions}, volume~44 of {\em American
  Mathematical Society Colloquium Publications}.
\newblock American Mathematical Society, Providence, RI, 1998.
\newblock With a preface in French by J. Tits.

\bibitem{Ojanguren_1974_locally_triv_Az_alg}
Manuel Ojanguren.
\newblock A nontrivial locally trivial algebra.
\newblock {\em J. Algebra}, 29:510--512, 1974.

\bibitem{Reiter_1975_Witts_extension_theorem}
Hans~J. Reiter.
\newblock Witt's theorem for noncommutative semilocal rings.
\newblock {\em J. Algebra}, 35:483--499, 1975.

\bibitem{Saltman_1978_Azumaya_algebras_w_involution}
David~J. Saltman.
\newblock Azumaya algebras with involution.
\newblock {\em J. Algebra}, 52(2):526--539, 1978.

\bibitem{Saltman_1981_Brauer_group_is_torsion}
David~J. Saltman.
\newblock The {B}rauer group is torsion.
\newblock {\em Proc. Amer. Math. Soc.}, 81(3):385--387, 1981.

\end{thebibliography}

\end{document}